\documentclass[11pt,english]{article}
\usepackage[T1]{fontenc}
\usepackage[utf8]{luainputenc}
\usepackage{xcolor}
\usepackage{pdfcolmk}
\usepackage{amsthm}
\usepackage{amsmath}
\usepackage{stmaryrd}
\usepackage{stackrel}
\usepackage{setspace}
\PassOptionsToPackage{normalem}{ulem}
\usepackage{ulem}
\usepackage{tikz}

\makeatletter

\providecolor{lyxadded}{rgb}{0,0,1}
\providecolor{lyxdeleted}{rgb}{1,0,0}

\numberwithin{figure}{section}
\numberwithin{equation}{section}
\numberwithin{table}{section}

\input amssymb.sty
\usepackage{etoolbox}
\patchcmd{\thebibliography}{\section*}{\section}{}{}

\usepackage{algpseudocode} \usepackage{algorithm}
\usepackage{algpseudocode,algorithm}
\usepackage{lipsum}
\usepackage{caption}
\usepackage{graphics}

\usepackage{amsmath}
\usepackage{amssymb}
\usepackage[all]{xy}

\usepackage{epsfig}\usepackage{youngtab}
	
	\newcommand{\ef}{\end{equation}}
\chardef\bslash=`\\ 

\newdir{:=}{{}}
\newcommand*\colvec[3][]{
	\begin{pmatrix}\ifx\relax#1\relax\else#1\\\fi#2\\#3\end{pmatrix}
}

\newtheorem*{thm*}{Theorem}

\newtheorem{lem}{Lemma}[section]
\newtheorem*{lem*}{Lemma}

\newtheorem*{corl*}{Corollary}

\newtheorem{prop}{Proposition}[section]
\newtheorem{prop*}{Proposition}

\theoremstyle{definition}
\newtheorem{defn}{Definition}[section]
\newtheorem{examp}{Example}
\newtheorem*{examp*}{Example}
\newtheorem*{remark*}{Remark}
\newtheorem*{CC*}{Crossover Conjecture}
\newtheorem*{Note*}{Note}
\newtheorem*{defn*}{Definition}

\theoremstyle{remark}
\newtheorem{remark}{Remark}[section]

\renewcommand{\sectionmark}[1]{}

\makeatother

\usepackage{babel}
\begin{document}
	
	\title{PERMUTATION OF EDGES IN MUTATION REDUCTION OF POINTED BRAUER TREES} 
	\author{Reut Frenkel-Mayzlish, Mary Schaps and Zehavit Zvi}
	
	\maketitle

\begin{abstract}

	Aihara developed an algorithm for Brauer tree algebras, which we call a mutation reduction,  for getting from a Brauer tree algebra to the simpler Brauer star algebra using a sequence  of mutations centered on edges.  Schaps and Zvi, using the  
	Schaps-Zakay theory of pointing the tree, showed that different algorithms for the sequence of mutations give permutations of the edges. 
	
	In \cite{K}, Kozakai gave a new algorithm for a mutation reduction that depends on a given pointing and describes the evolution of the pointing under the mutation reduction.
	
	In this paper, we define a pointed generalized Aihara algorithm and show that its permutation is the identity.  We give a general form for the permutations resulting from Kozakai's algorithm, which we illustrate with examples from  uni-branch binary trees.\footnote{MSC:20C05,20C20,16E35,18G80}

\end{abstract}

\section{INTRODUCTION}

\noindent This work concerns Brauer tree algebras, a widely studied class of 
algebras of finite representation type which includes all  blocks of cyclic defect 
group in modular group representation
theory. A block of cyclic defect group is a Brauer 
tree algebra and its Green correspondent is a Brauer star algebra.
Rickard proved \cite{R1}, \cite{R2} that every Brauer tree algebra has a tilting complex which 
makes it derived equivalent to the corresponding Brauer star algebra.  
Schaps-Zakay \cite{SZ1}, \cite{SZ2} showed that the tilting complexes in the 
opposite direction can be constructed from irreducible 
projective complexes of length two.  This is the 
all-at-once approach to the theory.

~

\noindent The other main approach is  the step-by-step approach
going back to König and Zimmermann {[}KZ1{]}, later formulated in terms
of mutations by Aihara \cite{Ai} used by Chan \cite{Ch}
and Zvonarevna \cite{Zv}, and extended by Kozakai \cite{K}.

\section{DEFINITIONS AND NOTATION}

\subsection{Pointed Brauer trees}

We first define  the Brauer
trees.

\noindent \begin{defn}
	
	\noindent Let $e$ and $m$ be natural numbers. A \textit{Brauer tree}
	of type\textit{ }$(e,m)$\textit{ }is a finite tree $(V,\ensuremath{{\cal E}})$
	where $V$ is the set of vertices, $\ensuremath{{\cal E}}$ is the
	set of edges, $|\ensuremath{{\cal E}|}=e$, together
	with a cyclic ordering of the edges at each vertex and a designation
	of an exceptional vertex which is assigned multiplicity $m$.
	
	\noindent \end{defn}

\noindent The set of all edges incident to vertex $u$ is denoted by $\ensuremath{{\cal E}}(u)$.
By \textquotedbl{}cyclic ordering\textquotedbl{} we mean that for
each edge \textit{E} in $\ensuremath{{\cal E}}(u)$ there is a `next'
edge in $\ensuremath{{\cal E}}(u)$ and that edge has a next edge
in $\ensuremath{{\cal E}(u)}$ etc., until each edge of $u$ is counted
exactly once, in which case\textit{ E} is the next one. We note that
if \textit{E} and $F$ are the only edges of $u$ then $F$ is next
after \textit{E} and \textit{E} is next after $F$.

\noindent \smallskip{}

\noindent Every Brauer tree can be embedded in the plane in such a
way that the cyclic ordering on each $\ensuremath{{\cal E}}(u)$ is
the counterclockwise direction.  Three important examples of Brauer trees are:

(i) The \textit{star} with the exceptional vertex in the center.

(ii) The \textit{linear tree}, which includes, for example, the Brauer trees of blocks of cyclic defect in the symmetric groups.

(iii) The \textit{uni-branch binary trees}, which have one branch at the exceptional vertex, and two new branches added at each vertex out to a distance $t$ from the exceptional vertex.

\noindent \smallskip{}

\noindent We relate Brauer trees to the structure of algebras. This is the main reason for their importance, although they are also used in Physics under the name Temperley-Leib algebras.

\noindent \begin{defn}An algebra $A$ is called a \textit{Brauer
		tree algebra} if there is a Brauer tree such that the indecomposable
	projective $A$-modules can be described by the following algorithm:\renewcommand{\labelenumi}{(\roman{enumi})}
	\begin{enumerate}
		\item There is a bijection between the edges of the tree and the isomorphism
		classes of simple $A$-modules, i.e. each edge is labelled by the
		corresponding isomorphism class. 
		\item If $S$ is a simple $A$-module and $P_{S}$ is the corresponding
		indecomposable projective $A$-module then $P_{S}\supseteq\text{rad}(P_{S})\supseteq\text{soc}(P_{S})\cong S$
		and $\text{rad}(P_{S})/\text{soc}(P_{S})$ is a direct sum of one
		or two uniserial modules corresponding to the two vertices of the edge, with composition factors determined 
		by a clockwise circuit around the vertex.  For edges at the exceptional vertex, 
		the clockwise circuit is made $m$ times.
	\end{enumerate}
	\noindent \end{defn}

\noindent  Even when the algebra which interests us is the block
of a group algebra, we will not use the actual block but rather its skeleton.

\noindent \begin{defn}Let $e$ and $m$ be natural numbers with $e>1$.
	Let $K$ be any field containing a primitive $e$th root of unity
	$\xi$. Let $\widehat{n}=em+1$. Let the cyclic group \({C_e} = \left\langle g \right\rangle \)
	act on the truncated polynomial ring $A=K[x]/x\hat{^{n}}$, $g:x\longmapsto\xi x$.
	The \textit{Brauer star algebra} of type $(e,m)$ is the skew group
	algebra $
	b=A[C_{e}],$
	in which $g$ and $x$ obey the relation
	$g^{-1}xg=\xi x.$
	The algebra $b$ has $e$ distinct simple modules, corresponding to
	the idempotents
	
	\[
	f_{i}=\frac{1}{e} \overset{e-1}{\underset{j=0}{\sum}}\xi^{-ij}g^{j},\quad i=1,..,e,
	\]

	\noindent and satisfying 
	$f_{i}x=xf_{i+1}.$
	
	~
	
	\noindent The corresponding indecomposable projective left modules
	are denoted by
	$P_{i}=bf_{i},\quad i=1,\ldots,e.$
	Each $P_{i}$ is uniserial, and the projective cover of
	rad$(P_{i+1})$ is $P_{i}$. We let $\left\{ x^{s}f_{i}\right\} _{s=0}^{em}$
	be a basis for $P_{i}$, and define the following maps:
	\begin{alignat*}{2}
		\varepsilon_{i}:P_{i} & \rightarrow P_{i}, & \varepsilon_{i}\left(f_{i}\right)=x^{e}f_{i}\\
		\\
		\tilde{h}_{ij}:P_{i} & \rightarrow P_{j},\thinspace\thinspace & \tilde{h}_{ij}\left(f_{i}\right)=x^{k}f_{i} & ,~~k\equiv j-i(mod\,e),\quad0\le k<e.
	\end{alignat*}

	\noindent For $i\neq j$, we denote $\tilde{h}_{ij}$ by $h_{ij}$,
	and for $i=j$ by $id$. For any $0\le\ensuremath{\ell}\leq m$
	we call a map $\varepsilon_{j}^{\ell}\tilde{h}_{ij}\left(=\tilde{h}_{ij}\varepsilon_{i}^{\ell}\right)$
	\textit{normal homogeneous} of degree $\ensuremath{\ell e+k}$, 
	
	\noindent where 
	\begin{align*}
		k & \equiv j-i\left(mod\text{~}e\right),\quad0\le k<e.
	\end{align*}
\end{defn}

\noindent \begin{defn}A cochain map $l_{\bullet}$ between $C^{\bullet}$ and $D^{\bullet}$
	is called \textit{normal homogeneous} if each vertical map is normal
	homogeneous.\end{defn}

\noindent \begin{defn}We call the homomorphism $\varepsilon_{i}^{m}:P_{i}\rightarrow P_{i}$
	the \textit{socle map}, for the obvious reason that it maps the top
	of $P_{i}$ into its socle \(\left\langle {{x^{em}}{f_i}} \right\rangle \).
	\noindent \end{defn}

\noindent The shift $T[n]$ of a complex $T$ shifts it $n$ degrees to the left if $n$ is positive,
and $\mid n\mid$  degrees to the right if $n$ is negative. A partial tilting complex $T$\textit{ }for the Brauer star
algebra $b$ is called \textit{two-restricted}$(PTC_{2})$ if it is
a direct sum of shifts of the indecomposable complexes

\noindent \begin{center}
	$\begin{array}{cccccccccccccc}
		S_{i}: &  &  &  &  &  &  & 0 & \rightarrow & P_{i} & \rightarrow & 0\\
		T_{ij}: &  &  &  &  &  &  & 0 & \rightarrow & P_{i} & \rightarrow & P_{j}& \rightarrow & 0,\quad i< j
		
	\end{array}$
	\par\end{center}

\noindent where the first nonzero  component of $S_{i}$
and $T_{ij}$ is in degree zero. The complexes $S_{i}[n]$ and $T_{ij}[n]$
are called \textit{elementary.} The map from $T_{ij}$ to $T_{ij}$
which is $\varepsilon_{i}^{m}$ on $P_{i}$ and zero on $P_{j}$ is
called the \textit{socle chain map}. It is chain homotopy equivalent
to the map which is zero on $P_{i}$ and $-\varepsilon_{j}^{m}$ on
$P_{j}$. A \textit{tilting complex} is a partial tilting complex from which every projective mordule can be obtained by taking mapping comes. A basis of the endomorphism ring of a tilting 
complex in $PTC_2$ is given by the normal homogeneous maps \cite{SZ1}.

~

\noindent \begin{defn}Let $G$ be a Brauer tree of type $(e,m)$.
	A \textit{pointing} $p$ on $G$ is the choice, for each nonexceptional
	vertex $u$, of a pair of edges $(i,j)$ which are adjacent in the
	cyclic ordering at $u$. If there is only one edge $i$ at $u$, then
	we take $(i,i)$ as the required pair. The tree $G$ together with
	a pointing $p$ is called a \textit{pointed Brauer tree} and denoted by $G(p)$.\end{defn}

\noindent \begin{remark}Recall that we have represented each Brauer
	tree by a planar embedding and the cyclic ordering at each vertex
	by counterclockwise ordering of the edges in the plane. We then represent
	the pointing $(i,j)$ by placing a point in the sector between edge
	$i$ and edge $j$ in a small neighborhood of $u$.\end{remark}

\noindent \begin{defn} Let $B$ be a Brauer tree with vertex set $V$.
	The distance $d(u)$ of any vertex $u\in V$ from the exceptional
	vertex $v$ is the number of edges in a minimal path from $u$
	to $v$ (and hence in any path without backtracking, since the
	graph is acyclic). For any edge $w$, the vertex closest to the exceptional 
	vertex will be called the \textit{near end} and the other vertex 
	will be called the \textit{far end}. The distance of an edge is the distance of the far
	end. 
\end{defn}

\noindent \begin{defn}Let $G$ be a Brauer tree with edge set $\ensuremath{{\cal E}}$.
	An \textit{edge numbering} of $G$ is a Brauer tree with all its edges
	numbered by $1,\ldots,e$. A \textit{vertex numbering} of $B$ is obtained
	from an edge numbering by giving the same number as the edge to the
	farthest vertex from the exceptional vertex on the edge. The exceptional
	vertex is numbered as 0.\end{defn}

~

\noindent \begin{defn}  A\textit{
		Green's walk} for a planar tree is a counter-clockwise circuit of the
	tree as if one were walking around the tree touching each edge with
	the left hand, starting with a chosen initial branch. A \textit{reversed Green's walk} 
	is a circuit in the opposite direction. Each pointing 
	and each choice of an initial branch determines an
	edge numbering  by starting at the exceptional vertex $v$
	and taking a Green's walk around the tree which begins with the initial branch, 
	and numbering the vertices and corresponding edges as $1, 2, 3,\dots, e$
	as one comes to the points. The \textit{reversed Green's walk} is a clockwise circuit of the tree, and produces a reversed edge numbering of the tree. \end{defn}

~
\noindent \begin{defn}
	At any vertex besides the exceptional vertex, we will call the first edge that one
	would meet on a Green's walk around the tree the \textit{ primary edge
	}of the vertex, and the first edge one would meet on a 
	reversed Green's walk will be called the \textit{coprimary edge}.
	The pointing which puts the point between the entering edge and the 
	primary edge at each vertex will be called the \textit{ordinary pointing}
	and the pointing which puts the point between the entering vertex and
	the coprimary edge will be called the  \textit{reversed pointing}.
	The \textit{left alternating pointing} has the point alternately on the left
	or right of the entering vertex, starting on the left for the edge connected 
	to the exceptional vertex. There is a corresponding right alternating pointing.
\end{defn} 

~

\noindent As described in \cite{SZ2}, each pointing determines a two-restricted star-to-tree tilting 
complex, in which the projectives of the tilting complex are taken from the 
Brauer star with the same $(e,m)$ and the opposite algebra of the 
endomorphism ring in the homotopy category is isomorphic to the Brauer tree algebra of the tree 
which was pointed. The inverse tree-to-star complex is described in \cite{RS}. The components $T_i$ of this star-to-tree
complex are stalk complexes for edges at the exceptional vertex
and complexes  $T_{ij}[-n_i]$ or $T_{ji}[-n_i+1]$, depending on 
whether the point is after or before $i$ in the cyclic ordering from the entering 
edge $j$. The shifts are adjusted so that every $P_i$ appears in a unique degree $n_i$. A different pointing would give a different tilting complex 
with isomorphic endomorphism ring.

\noindent \begin{defn}Consider a sequence $\{r_i\}_{i=1}^{l}$ of elements
	of $\left\{ 1,...,e\right\} $. Set
	\[
	h=\widetilde{h}_{r_{l-1}r_{l}}\circ...\circ\widetilde{h}_{r_{1}r_{2}}=\varepsilon_{r_{l}}^{\alpha}\widetilde{h}_{r_{1}r_{l}}.
	\]

	\noindent Then the sequence is $\textit{short}$ if $\alpha=0$ and
	$\textit{long}$ if $\alpha>0$. We generally represent the sequence
	in the form $r_{1}\rightarrow r_{2}\rightarrow...\rightarrow r_{l}$.\end{defn}

\noindent \begin{examp}If $e \geq 3$,
	\begin{itemize}
		\item $1\rightarrow 2\rightarrow 3$ is short
		\item $1\rightarrow 3\rightarrow 2$ is long.
	\end{itemize}
	\noindent \end{examp}

\noindent In \cite{SZ1} it was shown  that  
a chain map $\ell_\bullet:T_{ik} \rightarrow T_{jk}$  has the identity map at $P_k$
if  $i \rightarrow j \rightarrow k$ is short and is the socle map if
$i\rightarrow j \rightarrow k$ is long, and similarly for the dual
map from  $T_{ij}$ to $T_{ik}$.

\subsection{Mutation}

\noindent It is, of course, possible to define tilting complexes between
two general Brauer tree algebras. Of particular importance are the
tilting mutations of {[}Ai{]}, which go back to work of Rickard \cite{R2} and
Okuyama \cite{O}, or alternatively, to Kauer \cite{Ka}. Let $A$ be a finite dimensional basic algebra, with projective modules $P_j$. To each $j$,
we can associate an idempotent $\tilde{e}_{j}$ with $1_{A}$=$\stackrel[j=1]{e}{\sum}$$\tilde{e}_{j}$.

\noindent \begin{defn}
	Fix an $i$ and define $e_{0}=\underset{j\neq i}{\sum}$$\tilde{e_{j}}$.
	For any $j$$\ensuremath{{\cal \in E}}$ we define a complex by 
	
	$T_{j}^{(i)}=\begin{cases}
		\begin{array}{cccc}
			(0th) & \, & (1st) & \,\\
			P_{j} & \longrightarrow & 0 & \:\:j\neq i\\
			Q_{i} & \overset{\pi_{i}}{\longrightarrow} & P_{i} & \:\:j=i
		\end{array} & \,\end{cases}$
	
	\noindent where $Q_{i}\overset{\pi_{i}}{\longrightarrow}P_{i}$ is
	a minimal projective presentation of $\text{\ensuremath{\tilde{e}_{i}A}}/\tilde{e}_{i}Ae_{0}A$.
	Now we define $T^{(i)}:=\varoplus_{j\ensuremath{{\cal \in E}}}T_{j}^{(i)}$.
	The\textit{ mutation }$\mu_{i}^{+}$ of $A$ is $A'$ $\cong$ End$^{op}_{D^{b}(A)}$$T^{(i)}$.
	We will also consider the dual variant, as in $[S].$
	
	$T_{j}^{(i)}=\begin{cases}
		\begin{array}{cccc}
			(-1st) & \, & (0th) & \,\\
			0 & \longrightarrow &P_{j} & \:\:j\neq i\\
			P_{i} & \overset{\pi_{i}}{\longrightarrow} & Q_{i} & \:\:j=i
		\end{array} & \,\end{cases}$
	
	~
	
	\noindent where $Q_{i}$ is the minimal injective hull of the
	quotient of $P_{i}$ by the largest submodule containing only
	components isomorphic to the simple module  $S_{i}$.  This injective hull will be a direct
	sum of injective modules (which are also projective) whose irreducible
	socles give the socles of this quotient. We will denote this 
	by $\mu^{-} \cite{AI}$ (see, e.g., \cite{Zv} for more detail in the case of Brauer trees.)
	Since $A$ is a symmetric algebra, either version of the mutation will give a tilting complex. 
\end{defn}

\noindent Now let $A$ be a Brauer tree algebra.  Aihara showed in \cite{Ai} that there is a simple combinatorial
operation on edges $j$$\ensuremath{{\cal \in E}}$ which corresponds to the mutation. These also appeared in a paper by Kauer and are sometimes called ``Kauer moves''. We use the notation in \cite{K}, which is dual to the notation in \cite{Ai}.

\begin{figure}
	\[
	\xymatrix{
		&	&	&	&	&	& {\circ} \\
		& {\circ}   \ar@{-}[ul]^{l_1} \ar@{-}[ur]_{l_m} \ar@{}[u]|(0.4){\dots}	&	&	&	& {\circ} \ar@{-}[ur]^{h_1} \ar@{-}[dr]_{h_d} \ar@{}[r]|(0.4){\vdots}	&	\\
		{\circ}	&	& {\circ} \ar@{-}[rr]^i \ar@{-}[dl]^{k_1} \ar@{-}[ul]_{k_a} \ar@{}[l]|(0.4){\vdots}		&	& {\circ} \ar@{-}[ur]^{j_1} \ar@{-}[dr]_{j_b} \ar@{}[r]|(0.4){\vdots} &	& {\circ}  \\
		& {\circ} \ar@{-}[dl]^{g_1} \ar@{-}[ul]_{g_c} \ar@{}[l]|(0.4){\vdots}	&	&	&	& {\circ}  \ar@{-}[ld]_{n_1}  \ar@{-}[rd]^{n_l}  \ar@{}[d]|(0.4){\dots} \\
		{\circ}	& & & & & & & & & & & & & &\\ }	
	\]
	
	\caption{The center $i$ of a mutation}
\end{figure}
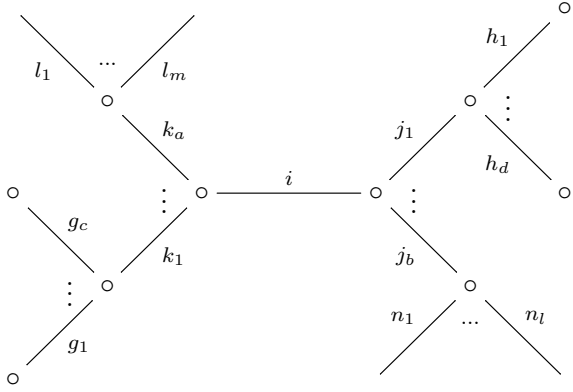

\begin{defn} For a mutation $\mu_i^-$ centered at edge $i$, we let $k_1$ and $j_1$ be the edges immediately preceding $i$ at the two vertices of $i$. We then erase $i$ and redraw it connecting the two farther vertices of $k_1$ and $j_1$ with each other, this new edge being called $i$ and all other names remaining the same. If the edge $i$ is a leaf,  then there is only one reattachment made.

$\mu-:$

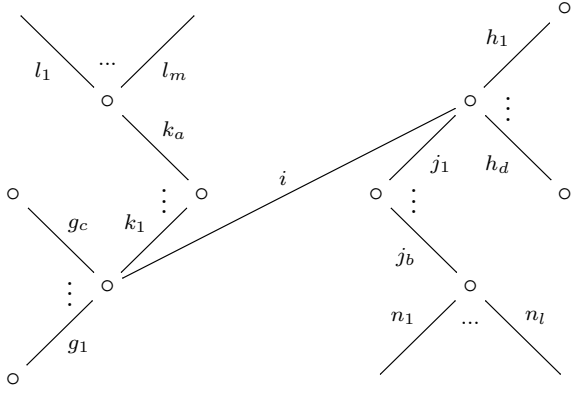
\xymatrix{
	&	&	&	&	&	& {\circ} \\
	& {\circ}	  \ar@{-}[ul]^{l_1} \ar@{-}[ur]_{l_m} \ar@{}[u]|(0.4){\dots}	 &	&	&	& {\circ} \ar@{-}[ur]^{h_1} \ar@{-}[dr]_{h_d} \ar@{}[r]|(0.4){\vdots}	&	\\{\circ}	 &	& {\circ} \ar@{-}[dl]_{k_1}   \ar@{-}[ul]_{k_a} \ar@{}[l]|(0.4){\vdots}	 & & {\circ} \ar@{-}[ur]_{j_1} \ar@{-}[dr]_{j_b}  \ar@{}[r]|(0.4){\vdots}	&	& {\circ} \\
	& {\circ} \ar@{-}[rrrruu]^i  \ar@{-}[dl]^{g_1} \ar@{-}[ul]_{g_c} \ar@{}[l]|(0.4){\vdots}	&	&	&	& {\circ}    \ar@{-}[ld]_{n_1}  \ar@{-}[rd]^{n_l}  \ar@{}[d]|(0.4){\dots} \\
	{\circ}	& & & & & & & & & & & & & &
}

The dual mutation  $\mu_j^+$ centered at edge $j$ is done by letting $k_a$ and $j_b$ be the edges immediately after $i$. The new edge $i$ will connect the farther vertex of $k_a$ with the farther vertex of $j_b$.

$\mu+:$

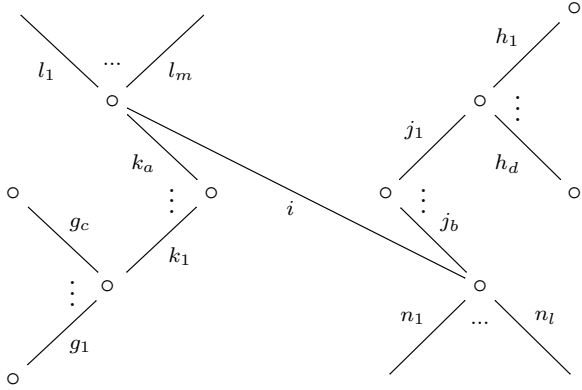
\xymatrix{
	&	&	&	&	&	& {\circ} \\
	& {\circ}   \ar@{-}[ul]^{l_1} \ar@{-}[ur]_{l_m} \ar@{}[u]|(0.4){\dots}	&	&	&	& {\circ} \ar@{-}[ur]^{h_1} \ar@{-}[dr]_{h_d} \ar@{}[r]|(0.4){\vdots}	&	\\{\circ}	 &	& {\circ} \ar@{-}[dl]^{k_1}   \ar@{-}[ul]^{k_a} \ar@{}[l]|(0.4){\vdots}		&	& {\circ}    \ar@{-}[ur]^{j_1} \ar@{-}[dr]^{j_b}  \ar@{}[r]|(0.4){\vdots}	&	& {\circ}	\\
	& {\circ} \  \ar@{-}[dl]^{g_1} \ar@{-}[ul]_{g_c} \ar@{}[l]|(0.4){\vdots}	&	&	&	& {\circ}  \ar@{-}[lllluu]^i    \ar@{-}[ld]_{n_1}  \ar@{-}[rd]^{n_l}  \ar@{}[d]|(0.4){\dots}    \\
	{\circ}	& & & & & & & & & & & & & &
}

\end{defn}

\section{MUTATION REDUCTION ALGORITHMS}

Assume we are given a Brauer tree $G$, with multiplicity $m$.  If $m>1$, 
then there is a designated exceptional vertex $v$. For $m=1$, we assume that one of
the vertices has been chosen as the exceptional vertex $v_0$. Since our graph is a tree,
there is a well-defined distance of each vertex $u$ from $v_0$ given by 
counting the number of edges on the unique path connecting them.
If the edges of the tree are labelled, then each vertex can be given
the same label as the first edge on this unique path.

~

\noindent \begin{defn}
	A \textit{mutation reduction} is a mutation or  sequence of mutations such that the distance of 
	each edge from the exceptional vertex does not ever increase,
	and such that at least one such distance actually decreases. A mutation
	reduction which ends at the Brauer star is called \textit{complete}.
\end{defn}

~

\noindent \begin{lem}\label{branch} \cite{SZv} Assume we are given a Brauer tree.
	
	\begin{enumerate}
		\item A mutation $\mu^-$ which is a mutation reduction must be centered at a primary edge.
		\item A mutation centered at a primary edge connected to an edge adjacent to the 
		exceptional vertex is a mutation reduction.
		\item After a complete mutation reduction, all the edges from a given branch 
		form an interval around the Brauer star, and these intervals follow the counterclockwise ordering of the branches.
	\end{enumerate}
\end{lem}

The dual (1) and (2) of this lemma for $\mu^+$ is given in \cite{Z2} as Lemma 3.2.  Part (3) is there replaced by a general lemma for mutation reduction, \cite{Z2}, Lemma 3.3.

\subsection{Aihara's algorithm}

\noindent We now give the original algorithm by Aihara \cite{Ai}

~

\noindent \uline{Aihara's Algorithm}\cite{Ai} 
\begin{enumerate}
	\item Choose an initial branch.  
	
	\item In a Green's walk starting at the root of the initial branch choose the first primary edge
	$w$ attached to  an edge adjacent to the exceptional vertex.  If the tree is not a star, there must
	be such an edge. 
	
	\item By Lemma \ref{branch}(2), the mutation centered on this edge $w$ is a mutation
	reduction, and from the proof we see that it creates two adjacent branches from the original,
	the first of which in counter-clockwise order is rooted at $w$.
	
	\item If $w$ was on the initial branch, let the new initial branch be the new branch rooted at $w$, and otherwise
	let the initial branch remain as before.  Begin again from 2.
	
\end{enumerate}

In \cite{SZv} it was shown that the Aihara algorithm is faster than an algorithm from \cite{Z} in which the centers are always leaves, but in fact in \cite{Z2}, the third author showed that all versions of Aihara's algorithm using different combinations of $\mu^-$ and $\mu^+$ are equally efficient in terms of the number of steps required for a complete mutation reduction:  The number of steps is always $e-1$ where $e$ is the number of edges, and this is the lowest possible for any mutation reduction.

\subsection{Kozakai's algorithm}

Kozakai constructed an algorithm for a pointed tree $G(p)$, that decides on the location of the center and whether or not it will be a primary or coprimary edge according to an existing pointing, and then determines the pointing of the result of the mutation. 

His algorithm depends on a pointed variant of the Kauer moves, the problem being to decide where to put the points at the vertices affected by the mutation. If $G(p)$ is a pointed Brauer tree, we define a pointed mutation $\mu^-(p)$,as follows.
\begin{enumerate}
	\item As an unpointed Brauer tree, $\mu^-(p)(G(p))= \mu^-(G)$.
	\item At every vertex not involved in the mutation, the point remains in the same segment as in $G(p)$.
	\item Refering to Fig. 2.1, if the point at one vertex of $i$  is between $i$ and one of the edges immediately before or after, then after the removal of $i$ it will be in the sector between those two edges.
	\item On the right-hand side of Fig. .1, if the point is in the sector between $g_1$ and $k_1$, then after the mutation it is between $g_1$ and $i$. The other side is dual. Heuristically, if we think of the vertex of $i$ as sliding down the edge before $i$, it pushes any points out of the way.
\end{enumerate}

\noindent The definition of $\mu^+(p)$ is dual.

\noindent \uline{Kozakai's algorithm}\cite{K}.
\noindent We assume we are given a pointed Brauer tree $G(p)$. We regard this tree as the $0$ step, letting 
$G_0=G$ and $p_0=p$.
\begin{enumerate}
	\item We fix a branch that is not a leaf as initial branch and designate it as the current branch.  Assume we have performed $t$ mutations, arriving at a Brauer tree $G_t$ with pointing $p_t$.
	\item If the point at the root of the current branch is not in the sector between the root and the primary edge, we let $\epsilon_t=-$ and if it is to the right, we let $\epsilon_t=+$. In the first case, we take a Green's walk consisting of primary edges until we reach the first point, which is at or before the first leaf we encounter.   We let the center $i_t$ be the  entering edge of the vertex at which the point is located. If the point at the root is in the sector between the root and the primary edge, we take a reversed Green's walk along coprimary edges until we reach the first point. By the choice of the sign, $i_t$ is at least at distance $1$ from the exceptional vertex. We now perform the mutation $\mu_{i_t}^{\epsilon_t}(p_t)$. Set $G_{t+1}=\mu^{\epsilon_t}G_{t+1}$ and let $p_{t+1}$ be the pointing of $G_{t+1}$ determined by $\mu_{i_t}^{\epsilon_t}(p_t)$. If $i_t$ is at distance $1$, so that the mutation creates two branches, the branch containing $i_t$ will be the current branch.
	\item Repeat Step 2 as long as the current branch is not a leaf. 
	\item When the current branch is a leaf, choose a new current branch  which is not a leaf and whose edges belong to the initial branch. Repeat Step 2 and Step 3 until all edges from the initial branch are leaves. If the tree is not yet a Brauer star, return to Step 1 to choose a new initial branch.
	\end{enumerate}
	
	\begin{remark}
		We made two adjustments to Kozkai's algorithm. The first was the return to Step 1 to choose a new initial branch when the first is exhausted.  This, we believe, was simply correcting an oversight, since otherwise the algorithm cannot continue.
		
		The second is a freer choice of new current branch in Step 4.  In the original algorithm, the new current branch is chosen by a Green's walk from the far clockwise side or a reversed Green's walk from the  counterclockwise side.  Unfortunately, when there are several branches, it could happen that on both of these walks one encounters a point at distance $1$ from the exceptional vertex.  In our version, the choice of parity $\epsilon_t$ is made on the basis of the point at the end of the root of the current branch. We have kept the spirit of the original algorithm by insisting that all edges of the initial branch be moved to the exceptional vertex before we choose a new initial branch.
	\end{remark}
	
	In Kozakai's original paper, he uses a random numbering for the edges in a branch, which are carried along during the mutation reduction to a set of leaves on the Brauer star which are all grouped together.  For $\mu^-$ this was proven in \cite{SZv} and for $\mu^+$, this was proven in \cite{Z2}.

Kozakai's algorithm arranges that in the mutation of the graph, there is never a point in the sector to which the edge 
$i$ will be connected.  Then, after the edges are reconnected, there is no point in the angle between the repositioned  edge and the edge to which it is now connected at the opposite end.  Otherwise, all points are left where they were.

\section{RANDOM NUMBERING AND POINTED NUMBERING}

Note that Kozakai numbers his edges at random, which makes the diagrams simpler.  Each mutation gives a one-to-one correspondence of edges, so at the end of a chain of mutations, there is a one-to-one correspondence of the original edges and the edges in the Brauer star, which have a counterclockwise ordering.

\begin{defn}
	For any branch of a Brauer tree, the edge connected to the exceptional vertex is called the $\textit{root}$.  All the edges before the point at the nonexceptional vertex of the root are called the \textit{left edge-group} and those after the point are called the \textit{right edge-group}.
	A branch at the exceptional vertex will be called a \textit{rooted branch}.
\end{defn}

\begin{lem}\label{split}
	If a pointed Brauer tree $G(p)$ is numbered according to the pointed edge numbering 
	from an interval $[a,\dots,b]$, and the root is numbered $c$, where $a \leq c \leq b$, then the permutation
	$\sigma$ has the form $\pi_L(c)\pi_r$, where $\pi_L$ is a permutation of 
	$a,a+1,\dots,c-2,c-1$ and $\pi_R$ is a permutation of $c+1,c+2,\dots, b-1,b$.
\end{lem}

\begin{proof}
	When we take a Green's walk around the tree, we encounter all the edges of the left edge-group before we reach the root point.  Since the root is given the number $c$, all the edges in the left edge-group must be numbered from $a,a+1,\dots, c-2,c-1$.
	The numbers remaining for the left edge-groups are $c+1,c+2,\dots,b-1,b$. Now, as long as the root point has not been exposed by the mutations, any operation on the initial branch is in the first case of Kozakai's algorithm, using $\mu^-$. If some mutation moves a non-trivial branch to the exceptional vertex, it will still be  clockwise from the root.  If we now begin to use $\mu_j^+$ on the initial branch, the new rooted branches we will create will lie on the clockwise side of the original root, either because they were created from the initial branch by $\mu_j^+$ or because they were created from some new rooted branch on the clockwise side and remain adjacent to the rooted branch from which they were taken.

	  In the Kozakai algorithm, the permutation $\sigma$ is a composition of transpositions all of which are transposition of two numbers in one edge group or the other. All the $\sigma_i$ on the left side are transposition of numbers of edges on the left edge-group.  All the $\sigma_i$ on the right side are transpositions of edges in the right edge-group.
	A product of transpositions from a certain set is a permutation of the members of that set. Thus, the final permutation has the structure asserted.  The cycles being disjoint, they commute, The order in which we wrote them was chosen to accord with the order of the numbers in the interval.
\end{proof}

\subsection{ The symmetry of Kozakai's algorithm}

Kozakai's algorithm appears to be non-symmetric, in that it choses the mutation 
$\mu^-$ and the counterclockwise Green's walk as default, and only applies the 
dual $\mu^+$ when the possibilities of using $\mu^-$ have been exhausted.  However, this apparent asymmetry is illusory.  

\begin{remark}\label{duality}
	In a uni-branch tree, if one changes the default in the Kozakai algorithm from $\mu^-$ to $\mu^+$ and changes the Green's walk to the reversed Green's walk,  one gets the same number of - and + mutations in the complete mutation reduction from the original centers which are connected to an edge adjacent to the exceptional vertex. The mutations change the placement of the points at the far end. Whether the total number of - and + mutations depends on the default bears investigating. 
\end{remark}

\section{THE POINTED GENERALIZED AIHARA ALGORITHM}

In \cite{Z2}, Zvi defined a generalization of the Aihara algorithm that used both primary and coprimary edges.  We now take Kozakai's idea of doing mutations on a pointed tree only at edges where there is no point between the entering edge and the center.  However,  Zvi also proves in that article that the generalized Aihara algorithms are more efficient in terms of number of centers than any other mutation reduction algorithm [\cite{Z2}, Thm. 3.1]. We want to define a mutation reduction algorithm for pointed trees based on Aihara. We will now define a pointed generalized Aihara algorithm and show that its corresponding permutation is the identity.

\begin{defn}
	Let $G(p)$ be a pointed tree.  In the \textit{pointed generalized Aihara algorithm}, we operate by $\mu^-$ on the primary edges connected to an edge adjacent to  the exceptional vertex as long as there is no point between the entering vertex and the primary edge.  We then operate on the coprimary edges 
	connected to an edge adjacent to the exceptional vertex by $\mu^+$ as long as there is no point.  At this stage, the original root has been reduced to a leaf because any branch to the right or left of the point would have been sent to the exceptional vertex already. 
	We then follow the same procedure for the new branches which have been created until we reach the Brauer star.

\end{defn}
\begin{prop}
		Let $G(p)$ be a pointed tree and do a complete mutation reduction according to the pointed generalized Aihara algorithm. Assume that the edges in the tree were numbered according to the edge numbering obtained from a Green's walk. Then the order of the edges in the Brauer tree will be in increasing numerical order in the counter-clockwise direction and thus the permutation associated with the mutation reduction algorithm will be the identity.  
\end{prop}

\begin{proof}
	Under the edge numbering determined by a pointing, all the numbers on a fixed brnach are consecutive, because we do not leave the branch until every edge has been given a number. Let us do induction on the maximal size $M$ of a rooted branch. When $M=1$, we have a Brauer star.  Assume the theorem for $M-1$. Given a rooted branch of size $M$, since $M>1$, there is at least one branch attached to the root.  If the point is to the left of the primary edge, then we operate by $\mu^-$, which attached the lopped off branch on the clockwise side of the original root.  If $[a,a+1,\dots, c-1]$ are the numbers in the left edge group, then the edge number $a$ must lie in the branch being removed,  since the branch represents an interval $[a,a+1,\dots,d]$, with $d<c$, By the induction hypothesis, a complete mutation reduction will produce  edges numbered $a, a+1, \dots, d$ in a counter-clockwise direction, followed by edges $d+1,\dots, c,c+1,\dots, b$.
	The case when there is no point between the entering edge and the coprimary edge is dual.  
	
\end{proof}

As an illustration of this proposition, we will study in detail a highly symmetric
class of pointed trees, of independent interest, the balanced unibranch binary trees.

\begin{defn} A \textit{unibranch binary tree} is a tree with a single root such that there are two branches going out of each non-exceptional vertex which is not at the end of a leaf. We call the tree full of order $t$ if all the leaves are at distance $t$ from the exceptional vertex.  If the unibranch binary tree is pointed, we call it \textit{balanced} if each point is between the two branches going out from the vertex.  
\end{defn}
We are going to show that the permutation of a full balanced unibranch binary tree is right-left symmetric.  In fact, by continuing the ideas of Prop. \ref{duality}, the permutation of any edge-numbered, unibranch symmetric tree with a symmetric pointing must be a left-right symmetric.  However, the notation for the general symmetric tree is awkward, so we are going to concentrate on the right-left symmetric full, balanced binary tree, for which the natural numbering is particularly well-structured and thus simplifies inductions.

We first collect some facts about full balanced unibranch binary trees.

\begin{lem}
	Let $G$ be a unibranch binary tree with a balanced pointing $p$, full of order $t$, numbered by the natural numbering.
	\begin{enumerate}
		\item $G$  has $2^t-1$ edges.
		\item The root is numbered by $2^{t-1}$.
		\item The edges at distance $k$ from the exceptional vertex, $1 \leq k \leq t$
		are numbered from left to right by $2^{t-k}(2n+1), n=1,2,\dots, 2^{k-1}-1$.
	\end{enumerate}
	Note that only the leaves at level $k=t$ are numbered by odd numbers.
\end{lem}

\begin{proof}
	\begin{enumerate}
		\item Since the number of edges doubles at each level, we get
		\[
		1+2+4+\dots+2^{t-1}=2^t-1.
		\]
		\item At the point of the root, the left edge-group and the right edge-group are both full unibranch binary trees of order $t-1$, so each contains $2^{t-1}$ edges.  Thus the root is numbered by the number after $2^{t-1}$, which is $2^{t-1}$, as desired.
		\item We now do induction on $k$. We have already established the base case of $k=1$. We assume that the numbering of the edges for $k<t$ is 
		$2^{t-k}(2n+1), n=1,2,\dots, 2^{k-1}-1$.  We calculate the numbers of the edges at level $k+1$. From the edge numbered by $g=2^{t-k}(2n+1)$ two edges go out. The branch on the left is lower than $g$ by the number of edges in its right edge-group plus one.  The right edge-group is a full binary tree of order
		$t-k-1$, and has $2^{t-k-1}-1$ branches. For the right edge, we add this same . Thus if the left edge has  number $g_L$ and the right edge has number $g_R$, we have
		\[
		g_L=g-2^{t-k-1}=2^{t-k-1}(4n+2-1)=2^{t-k-1}(4n+1),
		\]
		\[
		g_R=g+2^{t-k-1}=2^{t-k-1}(4n+2+1)=2^{t-k-1}(4n+3),
		\]
		This list of numbers is equivalent to $2^{t-k-1}(2m+1)$ for $m=0,1,\dots,2^k-1$, as desired.
	\end{enumerate}
\end{proof}

\begin{lem}
	The permutation $\sigma$ under the Kozakai algorithm  of a full unibranch binary tree of order $t$ with balanced pointing  is left-right symmetric. Letting $T=2^t-1$, there is a cycle $(i_1,i_2, \dots, i_h)$ in $\pi_L$ if and only if there is a cycle $(T-i_h,\dots, T-i_2,T-i_1)$ in $\pi_R$.
\end{lem}

\begin{proof}
	Since we showed in Prop. \ref{split} that the actions of $\mu_i^-$ and $\mu_j^+$ are independent and the order in which they are taken is irrelevant to the final permutation, we will follow each $\mu_i^-$ by the corresponding $\mu_j^+$, thus preserving the symmetry of the tree and of its pointing. 
	
	Start with the first primary node ending with a point, so that it is numbered by $1$. It is in fact a leaf, and by a sequence of $t-1$ mutations will become the leftmost leaf at the exceptional vertex.  The corresponding leaf which is coprimary will have the number $T-1$, the rightmost leaf, and a sequence of $t-1$ mutations $\mu_{T-1}^+ $ will make it the leftmost leaf in the Brauer star. The permutation $\pi_L$ contains a cycle $(1)$, while the permutation $\pi_R$ contains the dual symmetric cycle $(T-1)$. If $t=2$ for then $T=3$, we have reached the Brauer star, and the permucation is the identity, which is surely symmetric. Suppose $t>2$ and we will continue.
	
	Removing this leaf on the right exposed the point at the vertex $2$, which is at the far end of the edge numbered $2$ in level $t-1$, connectied to the primary edge numbered $3$ in level $t$.  According to Kozakai's algorithm, we perform the mutation $\mu_{2}^-$, creating a transposition $(2,3)$. The corresponding mutation $\mu_{T-2}^+$ where the edge $T-2$ is connected to the coprimary edge $T-3$ produces a transpostion $(T-3,T-2)$. Both edges have become leaves, which are moved to the exceptional vertex by appropriate powers of $\mu_3^-$ and $\mu_2^-$, and similarly on the right.  At this point the main branch is flanked by $1,3,2$ on the left and $T-2,T-3,T-1$ on the right. So far the permutation is right-left symmetric. If $t=3$ and $T=7$ we are finished, so asssume $T>3$.
	
	Removing the branch with edge numbered $2$ exposed the point at vertex $4$. The edge numbered $4$ is connected to the edge numbered $6$, which is attached to the edge $6$. Attached to $6$ are the two leaves numbered $5$ and $7$, as in Figure 5.1. After the mutations $\mu_4^-$, we have two possibilities.  
	
	If $t=4$, $T=15$, we get $4$ connected to the exceptional vertex at one end and to three leaves numbered $6$, $5$, and $7$, with the point between $5$ and $7$. The leaves $6$ and $5$ are moved to $v_0$.  Now, on the branch at $4$, the point on the Green's walk is on an edge attached to $v_0$, so we operate by $\mu_7^+$.  Performing the dual operations on the right, we have the branch rooted at $8$ stripped of all its leaves, and the sequence at $v_0$ is 
	\[
	1,3,2,6,5,4,7,8,9,12,11,10,14,13,15
	\]
	If $t>4$, then the mutation $\mu_4^+$ reattaches $4$ to $16$, and at the other end to $6$, $5$,$7$  as before.  Stripping off the leaves $6$ and $5$, we perform a second mutation
	$\mu_4^-$, after which the leaves $7$ and $4$ can be moved to the exceptional vertex, giving $1,3,2,6,5,7,4,$ to the left of the main branch.  the dual operations produce
	$T-4,T-7,T-5,T-6,T-2,T-3,T-1$ to the left.  These correspond to cycles $(2,3)(4,6,7)$ on the left and $(T-7, T-6,T-4)(T-3,T-2)$ on the right.
	
	We continued this far to give an example of a subbranch transplanted to $v_0$, for which one needs to use $\mu^+$.  However, as far as the proposition itself is concerned, the symmetry implies that we can match any action on the right by a dual action on the left, so that the total permutation is left-right symmetric.

\end{proof}

	\begin{tikzpicture}	

	\draw (0,0) node[anchor=south]{$v_0$} -- (0,-.5) node[anchor=east]{$8$} --(0,-1) node[anchor=north]{$\bullet$};
	\draw (0,-1) --(-1,-1.5) node[anchor=east]{$4$} -- (-2,-2) node[anchor=north]{$\bullet$};
	\draw (0,-1) -- (1,-1.5)node[anchor=west]{$12$} --(2,-2) node[anchor=north]{$\bullet$};
	\draw (-2,-2) -- (-2.5,-2.5)node[anchor=east]{$2$} -- (-3,-3)  node[anchor=north]{$\bullet$};
	\draw (-2,-2) -- (-1.5,-2.5) node[anchor=west]{$6$} -- (-1,-3) node[anchor=north]{$\bullet$}; 
	\draw (2,-2) --(2.5,-2.5) node[anchor=west]{$14$}  --(3,-3) node[anchor=north]{$\bullet$};
	\draw (2,-2) -- (1.5,-2.5) node[anchor=east]{$10$}  -- (1,-3) node[anchor=north]{$\bullet$};
	\draw (-3,-3) -- (-3.25,-3.5)node[anchor=east]{$1$} -- (-3.5,-4) node[anchor=north]{$\bullet$};
	\draw (-3,-3) -- (-2.75, -3.5) node[anchor=west]{$3$} -- (-2.5,-4) node[anchor=north]{$\bullet$};	
	\draw (-1,-3) -- (-.75,-3.5) node[anchor=west]{$7$} -- (-.5,-4) node[anchor=north]{$\bullet$};
	\draw (-1,-3) -- (-1.25,-3.5) node[anchor=east]{$5$} -- (-1.5,-4) node[anchor=north]{$\bullet$};	
	\draw (3,-3) -- (3.25,-3.5) node[anchor=west]{$15$} -- (3.5,-4) node[anchor=north]{$\bullet$};
	\draw (3,-3) -- (2.75,-3.5) node[anchor=east]{$13$} -- (2.5,-4) node[anchor=north]{$\bullet$};	
\draw (1,-3) -- (.75,-3.5) node[anchor=east]{$9$} -- (.5,-4) node[anchor=north]{$\bullet$};
\draw (1,-3) -- (1.25,-3.5)node[anchor=west]{$11
	$} -- (1.5,-4) node[anchor=north]{$\bullet$};	
\end{tikzpicture}
\begin{center}
	Fig.5.1
\end{center}

\begin{remark} In the original Master's thesis \cite{F}, it was shown that as long as $k \leq \frac{t}{2}$, there is a cycle in $\sigma$ of length $k+1$ of the form 
	\[
	(2^k,2^{k-1}(2^2-1),2^{k-2}(2^3-1),\dots, (2^{k+1}-1))
	\]
\end{remark}

{Department of Science Education, Weizmann Institute of Science,  Rehovot, Israel and Givat Washington Academic College, Givat Washington, Israel}\\
email: reutifrenkel@gmail.com\\
{Department of Mathematics, Bar-Ilan University, Ramat-Gan,  52900 Israel}.
email: {mschaps@macs.biu.ac.il}\\
{Department of Computer Science, Shamoon College of Engineering,  Be'er Sheva, 8410802 Israel}.\\
 email: zehavzv@sce.ac.il\\
MSC2020: {20C05,20C20, 16E35,18G80}\\
Keywords: {Brauer trees,  tilting complexes, mutations, algorithms}\\
Acknowledgements: This research was partly funded by a grant from the Research and Development Authority at Shamoon College of Engineering. It is based in part on the Master's thesis of Reut Frenkel-Mayzlish at Bar-Ilan University.

\end{document}